\documentclass[12pt]{article}
\usepackage{setspace} 
\usepackage[dvips]{graphicx} 
\usepackage{amsmath,amsthm,amsfonts}
\RequirePackage[colorlinks,citecolor=blue,urlcolor=blue,linkcolor=blue]{hyperref}
\hypersetup{
colorlinks = true,
citecolor=blue,
urlcolor=blue,
linkcolor=blue,
pdfauthor = {Alexey Kuznetsov},
pdfkeywords = {Levy process, Wiener-Hopf factorization, meromorphic function, infinite product, theta function},
pdftitle = {Wiener-Hopf factorization for a family of L\'evy processes related to theta functions},
pdfpagemode = UseNone
}
    \def\qed{\hfill$\sqcap\kern-8.0pt\hbox{$\sqcup$}$\\}
    \def\beq{\begin{eqnarray}}
    \def\eeq{\end{eqnarray}}
    \def\beqq{\begin{eqnarray*}}
    \def\eeqq{\end{eqnarray*}}

    \def\ee{\textnormal {e}}
    \def\re{\textnormal {Re}}
    \def\im{\textnormal {Im}}
    \def\p{{\mathbb P}}
    \def\e{{\mathbb E}}
    \def\r{{\mathbb R}}
    \def\c{{\mathbb C}}
    	
    \def\d{{\textnormal d}}
    \def\i{{\textnormal i}}
    \def\plus{{\scriptscriptstyle +}}
    \def\minus{{\scriptscriptstyle -}}
    \def\plusminus{{\scriptscriptstyle \pm}}
    \def\phiqp{\phi_q^{\plus}}
    \def\phiqm{\phi_q^{\minus}}
    
    \def\ind{{\mathbb I}}
	\newtheorem{theorem}{Theorem}
	
	\newtheorem{proposition}{Proposition}
	
	\newtheorem{definition}{Definition}

	\newtheorem{assumption}{Assumption}
	\newtheorem{remark}{Remark}

\title{Wiener-Hopf factorization for a family of L\'evy processes related to theta functions}
\author{A. Kuznetsov
\thanks{{Research supported by the
Natural Sciences and Engineering Research Council of Canada.  We are grateful to James Langley for  
helpful suggestions related to the proof of Theorem 1. We would also like to thank Andreas Kyprianou and Juan Carlos Pardo for many stimulating and insightful discussions.}}  \\ \\ 
Dept. of Mathematics and Statistics\\  York University
\\Toronto, ON, M3J 1P3 \\  Canada \\ e-mail: kuznetsov@mathstat.yorku.ca
 }\date{current version: July 24, 2010}

\begin{document}
\maketitle

\begin{abstract}
\bigskip
In this paper we study the Wiener-Hopf factorization for a class of L\'evy processes with double-sided jumps, characterized by the fact that the density of the L\'evy measure is given by an infinite series of exponential functions with positive coefficients.
 We express the Wiener-Hopf factors as infinite products over roots of a certain transcendental equation, 
and provide a series representation for the distribution of the supremum/infimum process evaluated at an independent exponential time. 
We also introduce five eight-parameter families of L\'evy processes, defined by the fact that the density of the L\'evy measure is a 
(fractional) derivative of the theta-function, and we show that these processes can have a wide range of behavior of small jumps. 
These families of processes are of particular interest for applications, since the characteristic exponent has a simple expression, 
which allows efficient numerical computation of the Wiener-Hopf factors and distributions of various functionals of the process.
\end{abstract}

{\vskip 0.5cm}
{\vskip 0.5cm}

\newpage

\section{Introduction}\label{intro}

Wiener-Hopf factorizaton and related fluctuation idenitities (see \cite{Bertoin}, \cite{Doney2007}, \cite{Kyprianou} and \cite{Sato}) allow us to study various functionals of a L\'evy process, such 
as extrema, first passage time, overshoots and undershoots, etc. 
 There is a growing number of applications of these functionals and of the Wiener-Hopf factorization techniques 
in many areas of Applied Probability, most prominently in Mathematical Finance and Insurance Mathematics. 
 However the number of processes with two-sided jumps for which one can obtain 
 explicit results on fluctuation identities is by no means large. In fact until very recent times the only known examples consisted of 
 a dense subclass of stable processes (see \cite{Doney1987}), processes with phase-type distributed jumps 
(see \cite{Asmussen2} or \cite{Asmussen} ), and more generally,
 processes having positive jumps with rational transform and arbitrary negative jumps (see \cite{Mordecki}). 
In the last several years there have appeared a number of new results
on the Wiener-Hopf factorization with varying degree of explicitness. First of all we would like to mention the class of Lamperti-stable processes, 
 which can be obtained by the Lamperti transformation (see \cite{Lamperti1972}) from positive self-similar Markov processes, which are related to stable processes. 
The family of Lamperti-stable processes is very interesting (and quite unique), since many fluctuation identities can be obtained in 
closed form in terms of elementary or special functions,
see  \cite{CaCha}, \cite{Chaumont2009}, \cite{KyPaRi} and the references therein. 
Second, in \cite{Kuznetsov2009} we have introduced 
a ten-parameter family of L\'evy processes, for which  the Wiener-Hopf factors can be identified in a semi-explicit form, more precisely as
infinite products involving solutions to a certain transcendental equation. It is the goal of this paper to build on our earlier ideas and 
 to extend the results in \cite{Kuznetsov2009}, as well as to provide new analytically tractable 
examples of L\'evy processes with semi-explicit Wiener-Hopf factorization.

The main idea in \cite{Kuznetsov2009} consists in the following observation: if the characteristic exponent $\Psi(z)$ can be extended to a meromorphic function (which means that its only singularities in the complex plane are poles) 
and if we have some additional information about the asymptotic behavior of the solutions of the equation $\Psi(z)+q=0$,
then we can obtain the Wiener-Hopf factors in essentially the same way as if $\Psi(z)$ was a rational function. 
Our first contribution in this paper is to show that the same technique works in a much more general setting: 
if the L\'evy measure is an infinite series of exponential 
functions with positive coefficients then the Wiener-Hopf factorization and the distribution of extrema can be computed rather explicitly in terms 
of the solutions of the equation $\Psi(z)+q=0$. As we will see, the following property plays the most important role in the proof of our main result:
the L\'evy measure is an infinite series of exponential functions with positive coefficients if and only if for every $q\ge 0$ the zeros and poles of
 $\Psi(z)+q$ lie on the imaginary axis and interlace. As our second contribution we present
several new families of L\'evy processes which satisfy the above condition and which will be interesting for applications due to their analytic tractability.

The paper is organized as follows: in section \ref{results} we prove Theorem \ref{thm1} on the Wiener-Hopf factorization and distribution of extrema, while  in section  \ref{section_examples} we construct five eight-parameter families of L\'evy processes, for which the density of the L\'evy measure is essentially 
the fractional derivative of the theta function. For each of these five families we provide explicit formulas for the characteristic exponent and asymptotic expressions for the large roots of $q+\Psi(z)=0$, which are important  
for the efficient implementation of numerical algorithms.

\section{Main Results}\label{results}

First let us present several definitions and notations which will be used in this paper. We define the open/closed upper half plane as
 \beqq
\c^{\plus}=\{z \in \c: \im(z)>0\}, \;\;\; \bar \c^{\plus}=\{z \in \c: \im(z)\ge 0\}, 
 \eeqq
and similarly for the negative half plane and for the positive/negative real half line.
We will study a one-dimensional L\'evy process $X$ started from zero, which is defined by the characteristic triple 
$(\mu,\sigma,\Pi)$.
The characteristic exponent $\Psi(z)=-\ln(\e [\exp(\i z X_1)])$ can be computed via the L\'evy-Khintchine formula (see \cite{Bertoin}) as follos
 \beq\label{def_Psi} 
 \Psi(z)=\frac 12 \sigma^2 z^2 -\i \mu z -\int\limits_{\r} \left(e^{\i z x}-1-\i z x h(x) \right) \Pi(\d x), \;\;\; z\in \r, 
 \eeq
where $h(x)$ is the cutoff function (everywhere in this paper we will use $h(x)\equiv 0$ or $h(x) \equiv 1$ as the measure $\Pi(\d x)$ will have exponential tails).
 We define the supremum/infimum processes as
\beqq
S_t=\sup\limits_{0\le s \le t} X_s, \;\;\; I_t=\inf\limits_{0\le s \le t} X_s.
\eeqq
Wiener-Hopf factors (see \cite{Bertoin}, \cite{Doney2007}, \cite{Kyprianou} or \cite{Sato}) are defined as
\beqq 
\phiqp(z)=\e\left[e^{\i z S_{\ee(q)}}\right], \; \textnormal{ for } \; z\in \bar \c^{\plus},  \;\;\;
 \phiqm(z)=\e\left[e^{\i z I_{\ee(q)}}\right], \; \textnormal{ for } \; z\in \bar \c^{\minus},
\eeqq
 where  the random variable $\ee(q)$ is  exponentially distributed  with parameter $q>0$ and is independent of the process $X$.

In order to specify the L\'evy measure $\Pi(\d x)$ we start with the
 four sequences of positive numbers $\{a_n,\rho_n,\hat a_n,\hat\rho_n\}_{n\ge 1}$, and assume that the sequences 
$\{\rho_n\}_{n \ge 1}$ and $\{\hat \rho_n\}_{n\ge 1}$ are stricly increasing, and that $\rho_n \to +\infty$ 
and $\hat \rho_n \to +\infty$ as $n\to +\infty$. 

\begin{assumption}\label{assumption_an_rhon}
Series $\sum_{n\ge 1 } a_n \rho_n^{-2}$ and $\sum_{n\ge 1 } \hat a_n \hat \rho_n^{-2}$ converge.
\end{assumption}

Next we define the function $\pi(x)$ as 
\beq\label{def_nu_ab}
\pi(x)=\ind(x>0) \sum\limits_{n\ge 1} a_n \rho_n e^{-\rho_n x}+ \ind(x<0) \sum\limits_{n\ge 1} \hat a_n \hat \rho_n e^{\hat \rho_n x}.
\eeq
It is easy to see that the above series converges for all $x \ne 0$, uniformly in $x$ on $\r\setminus (-\epsilon,\epsilon)$ for all $\epsilon>0$. 
To check this one should use Assumption \ref{assumption_an_rhon} and the inequality $\exp(-\rho_n \epsilon)<\rho_n^{-2}$, which holds for all $n$ large enough.
It is clear that $\pi(x)$ is a positive function which decays exponentially as $|x|\to \infty$, and as we show in the next Proposition, Assumption
\ref{assumption_an_rhon} guarantees that $\pi(x)$ can be used to define a L\'evy measure.

\begin{proposition}\label{prop_nu}
Assumption \ref{assumption_an_rhon} implies that 
  $\int_{\r} x^2 \pi(x) \d x < \infty$.
\end{proposition}
\begin{proof}
As we have mentioned above, for all $\epsilon>0$ the first series in (\ref{def_nu_ab}) converges uniformly for $x \in (\epsilon,\infty)$, 
 thus we can integrate term by term and obtain
 \beqq
  \int\limits_{\epsilon}^{\infty} x^2  \pi(x) \d x&=&\sum\limits_{n\ge 1} a_n \rho_n 
  \int\limits_{\epsilon}^{\infty} x^2 e^{- \rho_n x} \d x=
  \sum\limits_{n\ge 1} a_n \rho_n^{-2} 
  \int\limits_{\rho_n \epsilon}^{\infty} u^2 e^{-u} \d u \\ &<& \sum\limits_{n\ge 1} a_n \rho_n^{-2} 
  \int\limits_{\rho_1 \epsilon}^{\infty} u^2 e^{-u} \d u < 2\sum\limits_{n\ge 1} a_n \rho_n^{-2},
 \eeqq
where in the second step we have changed the variable of integration $x \mapsto u=\rho_n x$. 
We see that the integral on the left-hand side of the above inequality increases and is bounded as $\epsilon \to 0^+$, therefore it converges.
The convergence of the integral over $(-\infty,0)$ is proved similarly.
\end{proof}

Proposition \ref{prop_nu} allows us to use $\pi(x)$ as the density of a L\'evy measure $\Pi(\d x)=\pi(x) \d x$. 
Note that since  $\pi(x)$ decays exponentially as $x\to \infty$ we can use the cutoff function $h(x)\equiv 1$ in (\ref{def_Psi}); 
this will be the default choice everywhere in this section. The above property also allows us to work with the Laplace exponent, defined as 
\beq\label{def_Laplace_exponent}
\phi(z)=\ln\left(\e \left[e^{z X_1} \right]\right)=-\Psi(-\i z).
\eeq

\begin{proposition}\label{prop_eq_Psi}
The Laplace exponent $\phi(z)$ is a real meromorphic function which has the following partial fraction decomposition
 \beq\label{eq_Psi}
 \phi(z)=\frac 12 \sigma^2 z^2 + \mu z
 +z^2 \sum\limits_{n\ge 1}\frac{a_n }{\rho_n(\rho_n- z)}+z^2 \sum\limits_{n\ge 1} \frac{\hat a_n}{\hat \rho_n (\hat \rho_n+z)}
, \;\;\; z\in \c.
 \eeq
\end{proposition}
\begin{proof}
 Use  (\ref{def_Psi}), (\ref{def_nu_ab}) and (\ref{def_Laplace_exponent}), integrate term by term and rearrange the resulting infinite series.
\end{proof}

\begin{proposition}\label{prop_localization}
Assume that $q>0$.  Equation $\phi(z)=q$ has solutions $\{\zeta_n,-\hat \zeta_n\}_{n\ge 1}$, where $\{\zeta_n\}_{n\ge 1}$
and $\{\hat \zeta_n\}_{n\ge 1}$ are sequences of positive numbers which satisfy the following interlacing
property
 \beq\label{eq_localization}
 && 0< \zeta_1 < \rho_1 < \zeta_2 < \rho_2 < \dots \\ \nonumber
 && 0< \hat \zeta_1 < \hat \rho_1 < \hat \zeta_2 < \hat \rho_2 < \dots 
 \eeq 
\end{proposition}
\begin{proof}
Using (\ref{eq_Psi}) we rewrite the equation $\phi(z)=q$ as 
 \beq\label{eq_Psi_iz}
  z^2 \sum\limits_{n\ge 1}\frac{a_n }{\rho_n(\rho_n- z)}+z^2 \sum\limits_{n\ge 1} \frac{\hat a_n}{\hat \rho_n (\hat \rho_n+z)}= 
 q - \frac 12 \sigma^2 z^2 - \mu z.
 \eeq
Denote the left-hand side of (\ref{eq_Psi_iz}) as $\phi_1(z)$ and the right-hand side as $\phi_2(z)$. 
First let us  check that there exists a solution to (\ref{eq_Psi_iz}) on intervals $(0,\rho_1)$ and $(-\hat \rho_1,0)$.  
We observe that $\phi_1(0)=0$  and
$\phi_1(z) \nearrow +\infty$ as $z \nearrow \rho_{1}$ or as  $z \searrow -\hat \rho_{1}$; at the same time function $\phi_2(z)$ is continuous and
 $\phi_2(0)>0$, all that is left to do is to apply the Intermediate Value Theorem. Other intervals can be verified in a similar way.
\end{proof}

It is important to emphasize that Proposition  \ref{prop_localization} does not claim that the sequences $\{\zeta_n, -\hat \zeta_n\}_{n\ge 1}$ include {\it all}
solutions to $\phi(z)=q$. The statement is that {\it some} solutions are real and that they interlace with the poles $\{\rho_n, -\hat \rho_n\}_{n\ge 1}$. As we will see later, it is in fact true 
that there are no other solutions, but the proof of this statement is not trivial and requires some deep results from the theory of meromorphic functions.  
.

Our first main result in this paper is the following Theorem, which identifies the Wiener-Hopf factors and the distribution of extrema of the process $X$.
\begin{theorem}\label{thm1}
Assume that $q>0$. Then for $\re(z)>0$ 
\beq\label{eq_WH}
 \phiqp(\i z)=\e \left[ e^{-z S_{\ee(q)}} \right]= \prod\limits_{n\ge 1} \frac{1+\frac{z}{\rho_n}}{1+\frac{z}{\zeta_n}}, \;\;\;
\phiqm(-\i z)=\e \left[ e^{z I_{\ee(q)}} \right]= \prod\limits_{n\ge 1} \frac{1+\frac{z}{\hat \rho_n}}{1+\frac{z}{\hat \zeta_n}}.
\eeq
The distribution of $S_{\ee(q)}$ can be identified as an infinite mixture of exponential distributions
 \beq\label{eq_Stau} 
\p(S_{\ee(q)}=0)=c_0, \qquad \frac{\d }{\d x} \p(S_{\ee(q)} \le x)=
 \sum\limits_{n \ge 1} c_n \zeta_{n} e^{-\zeta_{n}x}, \;\;\; x>0,
\eeq
where the coefficients $\{c_n\}_{n\ge 0}$  are positive, satisfy $\sum_{n\ge 0} c_n=1$, and can be computed as
\beq\label{def_coeff_c0_cn}
c_0=\lim_{n\to +\infty}\prod\limits_{k=1}^n \frac{\zeta_k}{\rho_k}, \;\;\;
c_n=\left(1-\frac{\zeta_n}{\rho_n}\right) \prod\limits_{\substack{k\ge 1 \\ k\ne n}}  \frac{1-\frac{\zeta_n}{\rho_k}}{1-\frac{\zeta_n}{\zeta_k}}.
\eeq
The distribution of $-I_{\ee(q)}$ has the same form as above, with $\{\rho_n, \zeta_n\}$ replaced by $\{\hat \rho_n, \hat \zeta_n\}$. 
\end{theorem}

Our proof of the Theorem \ref{thm1} is based on the two important results from the theory of meromorphic functions, 
which we present below.  These two theorems combined together state that the L\'evy measure has the form (\ref{def_nu_ab}) if and only if for every $q\ge 0$ the zeros and poles of the real meromorphic function $\phi(z)-q$ are real and interlace, and this function
can be represented as an infinite product.
\begin{theorem}{(\cite{Levin1996}, page 220, Theorem 1)}\label{thm_Levin}
 A real meromorphic function $f(z)$ satisfies $f(z) \in \c^{\plus}$ for all $z \in \c^{\plus}$ if and only if it can be represented in the form
\beq\label{Levin_inf_product}
f(z)=c \frac{z-a_0}{z-b_0} \prod\limits_{n\in {\mathbb Z}\setminus\{0\}} \frac{1-\frac{z}{a_n}}{1-\frac{z}{b_n}},
\eeq   
where $c>0$, the zeros $a_n$ and the poles $b_n$ are real,  satisfy the interlacing property
\beqq
b_n < a_n < b_{n+1},\;\;\; n\in {\mathbb Z},
\eeqq
 and $a_{-1} < 0 < b_1$. 
\end{theorem}
\begin{theorem}{(\cite{Chebotarev1949}, page 197, Theorem 1)}\label{thm_Chebotarev}
  A real meromorphic function $f(z)$ satisfies $f(z) \in \c^{\plus}$ for all $z \in \c^{\plus}$ if and only if it can be represented in the form
 \beq\label{Cheb_part_fractions}
 f(z)=\alpha z+\beta+\frac{B_0}{b_0-z}+\sum\limits_{n \in {\mathbb Z}\setminus\{0\}} B_n \left[ \frac{1}{b_n-z}-\frac{1}{b_n} \right],
 \eeq
where the poles $b_n$ are real, $b_n<b_{n+1}$ and $b_{-1}<0<b_1$,  $\alpha \ge 0$, $\beta \in \r$,  $B_n \ge 0$ for $n\ge 0$ and the series $\sum B_n b_n^{-2}$ converges.
\end{theorem}

{\it Proof of Theorem \ref{thm1}:}
First let us prove that the infinite products in (\ref{eq_WH}) converge. The product $\prod b_n$ converges if and only if the series $\sum (b_n-1)$ converges, thus we need to consider the series 
\beq\label{proof1}
\sum\limits_{n\ge 1} \left[ \frac{1+\frac{z}{\rho_n}}{1+\frac{z}{\zeta_n}} -1 \right]=
z \sum\limits_{n\ge 1} \frac{1}{1+\frac{z}{\zeta_n}} \left( \frac{1}{\rho_n} - \frac{1}{\zeta_n}\right). 
\eeq
Using the interlacing property (\ref{eq_localization}) we find that
\beqq
0 < \dots < \rho_2^{-1} < \zeta_2^{-1} < \rho_1^{-1} < \zeta_1^{-1},
\eeqq
thus we obtain
\beqq
 0< \sum\limits_{n\ge 1} \left( \frac{1}{\zeta_n} - \frac{1}{\rho_n}\right)< \sum\limits_{n\ge 1} \left( \frac{1}{\zeta_n} - \frac{1}{\zeta_{n+1}}\right)=\zeta_1^{-1},
\eeqq
and conclude that the series in the right-hand side of (\ref{proof1}) converges, which guarantees the convergence of the infinite products in (\ref{eq_WH}).

Next, let us establish the following infinite product factorization
\beq\label{eq_main_factorization}
\frac{q}{q-\phi(z)}=\prod\limits_{n\ge 1} \frac{1-\frac{z}{\rho_n}}{1-\frac{z}{\zeta_n}}
\prod\limits_{n\ge 1} \frac{1+\frac{z}{\hat \rho_n}}{1+\frac{z}{\hat \zeta_n}}.
\eeq
To prove this factorization, we use (\ref{eq_Psi}) and rewrite the function $(\phi(z)-q)/z$ in the following form
\beqq
\frac{\phi(z)-q}{z}=\frac 12 \sigma^2 z + \mu - \frac{q}{z}
 + \sum\limits_{n\ge 1}a_n \left( \frac{1}{\rho_n-z} - \frac{1}{\rho_n} \right) 
+ \sum\limits_{n\ge 1} \hat a_n \left( \frac{1}{-\hat \rho_n-z} - \frac{1}{-\hat \rho_n} \right).
\eeqq
Using the above equation and Theorem \ref{thm_Chebotarev} we conclude that function $(\phi(z)-q)/z$ maps the upper half plane into itself,
therefore applying Theorem \ref{thm_Levin} and Proposition \ref{prop_localization} we find that 
\beqq
\frac{\phi(z)-q}{z}=c \frac{z-\zeta_1}{z} \prod\limits_{n\ge 1} \frac{1-\frac{z}{\zeta_{n+1}}} {1-\frac{z}{\rho_n}}
 \prod\limits_{n\ge 1} \frac{1+\frac{z}{\hat \zeta_{n}}} {1+\frac{z}{\hat\rho_n}},
\eeqq 
which in turn implies (\ref{eq_main_factorization}). 

Let us introduce functions
\beq\label{def_f}
 f(z)= \prod\limits_{n\ge 1} \frac{1+\frac{z}{\rho_n}}{1+\frac{z}{\zeta_n}}, \;\;\;
 \hat f(z)=\prod\limits_{n\ge 1} \frac{1+\frac{z}{\hat \rho_n}}{1+\frac{z}{\hat \zeta_n}}.
\eeq
Using Theorem \ref{thm_Levin} we find that the real meromorphic function $zf(z)$ maps the upper halfplane into itself, 
thus again we apply Theorem \ref{thm_Chebotarev} to obtain that $zf(z)$ has the partial fraction decomposition (\ref{Cheb_part_fractions}), 
which implies that $f(z)$ has the partial fraction decomposition
 \beq\label{f_z_part_fractions}
 f(z)= c_0 + \sum\limits_{n\ge 1} \frac{c_n \zeta_n}{\zeta_n+z},   
 \eeq 
 where $c_n\ge 0$ for $n\ge 0$ and the series $\sum c_n$ converges. Setting $z=0$ in the above equation we find that $\sum c_n=1$, and therefore 
 $f(z)$ is the Laplace transform of the mixture of exponential distributions, which implies that $f(z)$ is the Laplace transform of a positive infinitely
divisible random variable with zero drift. The same result holds for $\hat f(z)$. Equation (\ref{eq_main_factorization}) tells us that 
$q/(q-\phi(z))=f(-z) \hat f(z)$, and using the uniqueness of the Wiener-Hopf factorization (see \cite{Sato}) we conclude
that $\phiqp(\i z)=f(z)$ and $\phiqm(-\i z)=\hat f(z)$, thus obtaining (\ref{eq_WH}).  

In order to finish the proof we only need to establish formulas (\ref{def_coeff_c0_cn}). The expression for $c_0$ follows from the following computations
 \beqq
  \p(S_{\ee(q)}=0)=\lim\limits_{z\to +\infty} \e \left[ e^{-z S_{\ee(q)}} \right]=\lim\limits_{z\to +\infty}
   \lim\limits_{N\to +\infty} \prod\limits_{n=1}^N \frac{1+\frac{z}{\rho_n}}{1+\frac{z}{\zeta_n}}=\lim_{N\to +\infty}\prod\limits_{n=1}^N \frac{\zeta_n}{\rho_n},
 \eeqq
where in the last step we have interchanged the two limits, which is possible due to the fact that the limit
\beqq
\lim\limits_{z\to +\infty} \prod\limits_{n=1}^N \frac{1+\frac{z}{\rho_n}}{1+\frac{z}{\zeta_n}}
\eeqq
converges uniformly in $N$. The expression for $c_n$ in (\ref{def_coeff_c0_cn}) follows from (\ref{def_f}) and 
the fact that $c_n\zeta_n$ is the residue of $f(z)$ at the pole $z=-\zeta_n$.  
\qed

\section{Examples}\label{section_examples}


Theorem \ref{thm1} gives us the Wiener-Hopf factors and the distribution of $S_{\ee(q)}$ and $I_{\ee(q)}$ for any L\'evy process
whose L\'evy measure is an infinite series of exponential functions with positive coefficients. 
All the formulas in Theorem \ref{thm1} are based only on the zeros and the poles of the 
meromorphic function $\phi(z)-q$. The poles $\{\rho_n, -\hat \rho_n\}_{n\ge 1}$ are usually known explicitly, 
 but in order to find the zeros $\{\zeta_n, -\hat \zeta_n\}_{n\ge 1}$ one has to solve the transcendental equation $\phi(z)=q$, and this 
has to be done numerically. 
It is clear that if  we have to rely on the partial fraction decomposition
(\ref{eq_Psi}) in order to evaluate $\phi(z)$, such numerical computations will be quite challenging, if not impossible. 
Thus it is very important to find families of L\'evy processes, for which the L\'evy measure is an infinite series of exponentials,
and at the same time the Laplace exponent can be computed in closed form. In \cite{Kuznetsov2009} we have introduced the $\beta$-family: a ten-parameter family of L\'evy processes, for which the characteristic exponent can be computed in terms of the beta function. A very special subclass of the $\beta$-family
is a five-parameter family of L\'evy processes (see Section 3 in \cite{Kuznetsov2009}), whose jump component behaves similarly to the normal inverse Gaussian process (see \cite{Barndorff1997}), and the characteristic exponent is expressed in
terms of elementary (trigonometric) functions. In this section we present five parametric families of L\'evy processes, which have a number of desirable properties: the characteristic exponent is expressed in terms of rather simple functions, such as trigonometric functions or the digamma function $\psi(z)=\Gamma'(z)/\Gamma(z)$ (see \cite{Jeffrey2007} for the definition and properties of the digamma function). The density of the L\'evy measure decays exponentially at infinity and has
 a singularity at zero 
 \beqq
 \pi(x) \sim a_{\plusminus} |x|^{-\chi}, \;\;\; x\to 0^{\plusminus}, 
 \eeqq
of the order $\chi \in \{\frac12,1,\frac32,2,\frac52\}$, thus ``covering'' the complete range of admissible singularities $\chi \in (0,3)$. 
In particular, when $\chi = 1$ we have a process
of infinite activity of jumps but of finite variation, whose jump part is similar to the Variance Gamma process (see \cite{Madan1990}), and when $\chi=2$ we obtain a process with infinite variation of jumps, whose jump part is 
similar to the normal inverse Gaussian process.

To define these processes we introduce the function $\Theta_k(x)$ as follows
\beq\label{def_Thetak}
\Theta_{k}(x)=\delta_{k,0}+2\sum\limits_{n\ge 1} n^{2k} e^{-n^2 x}, \;\;\; x>0,
\eeq
where $\delta_{k,0}=1$ if $k=0$ and $\delta_{k,0}=0$ otherwise. Note that $\Theta_0(x)=\theta_3(0,e^{-x})$ (see \cite{Jeffrey2007} for the definition and properties of the theta functions $\theta_i(z;\tau)$), 
and for $k \in {\mathbb N}$ the function $\Theta_k(x)$ is just the $k$-th order  derivative of the theta function $\theta_3(0,e^{-x})$. 
In fact, results similar to the ones presented in this section can be established if we use $\theta_2(0,e^{-x})$ instead of  $\theta_3(0,e^{-x})$.

\begin{definition}\label{def_nu_chi}
For $0< \chi  <3$ and $x\ne 0$ we define
\beqq
\pi_{\chi}(x)= \ind (x>0) c_1\beta_1 e^{-\alpha_1 x}\Theta_{k}(x\beta_1) + \ind(x<0) c_2 \beta_2 e^{\alpha_2 x}\Theta_{k}(-x\beta_2) ,
\eeqq
where $c_i, \alpha_i, \beta_i >0$ and $k=\chi-1/2$.
\end{definition}

\begin{proposition}\label{prop_nu_chi}
 Function $\pi_{\chi}(x)$ has the following asymptotics
\beqq
 \pi_{\chi}(x) \sim \left\{ 
\begin{array}{ll} \Gamma(\chi) c_1\beta_1^{1-\chi} |x|^{-\chi}, \;\;\; x\to 0^+,\\
		  \Gamma(\chi) c_2\beta_2^{1-\chi} |x|^{-\chi}, \;\;\; x\to 0^-.
\end{array}
\right.
 \eeqq
\end{proposition}
\begin{proof}
Let $h=\sqrt{x}$. Then using the fact that a definite integral is the limit of the Riemann sum we obtain that as $h \to 0^+$ 
 \beqq
 \sum\limits_{n\ge 1} n^{2k} e^{-n^2 x}= h^{-1-2k}\left[h\sum\limits_{n\ge 1} (hn)^{2k} e^{-(hn)^2 }\right]=x^{-\frac12-k} 
 \left[ \int\limits_0^{\infty} y^{2k} e^{-y^2}dy + o(1) \right].
 \eeqq
To finish the proof use Definition \ref{def_nu_chi} and the above asymptotic relation.
\end{proof}

Proposition \ref{prop_nu_chi} and Definition \ref{def_nu_chi} guarantee that $\int_{\r} x^2 \pi_{\chi}(x) \d x$ exists, thus $\pi(x)$ can be used to define the
 density of a L\'evy measure $\Pi_{\chi}(\d x)=\pi_{\chi}(x) \d x$. We define the L\'evy process $X$ using the the characteristic triplet $(\mu,\sigma,\Pi_{\chi})$
and the L\'evy-Khinchine formula (\ref{eq_Psi}).  

\begin{remark}
Note that in a similar way one can construct L\'evy processes with asymmetric behavior of small positive/negative jumps. One should  
extend Definition \ref{def_nu_chi} to allow parameter $k_1$ \{$k_2$\} to control the order of the singularity of $\pi(x)$ as $x\to 0^+$ \{$x\to 0^-$\}. 
\end{remark}

In the next five sections we will present the results for the L\'evy process $X$ with $\chi \in \{\frac12,1,\frac32,2,\frac52\}$. 
The explicit formulas for the characteristic exponent $\Psi(z)$ are derived using the following approach: 
first we obtain the characteristic exponent $\Psi(z)$ for $\chi=\frac12$ (this is just a simple application of formulas 6.162 in \cite{Jeffrey2007}) and for $\chi=1$, in which case we have to use the following series expansion for the digamma function (see  formula 
8.362.1  in \cite{Jeffrey2007})
 \beqq
  \psi(x)=-\gamma-\frac1{x}+x\sum\limits_{n\ge 1} \frac{1}{n(n+x)}.
 \eeqq
All the other cases, when $\chi \in \{\frac32,2,\frac52\}$, can be easily derived from the above two by using Definition \ref{def_nu_chi}, 
the fact  that $\Theta_{k+1}(x)=-\frac{\d}{\d x} \Theta_k(x)$ (which follows from  (\ref{def_Thetak})) and applying the following Proposition,
(which can be easily established by integration by parts).
\begin{proposition}
${ }$
 \begin{itemize}
  \item [(i)] Assume that $\int_{\r^{\plus}} e^{-\alpha x} \pi(x) \d x < \infty$. Then for $z\in \r$
 \beqq
  \int\limits_{\r^{\plus}} \left(e^{\i z x}-1\right) e^{-\alpha x} \pi'(x) \d x=(\alpha-\i z) f_1(z)-\alpha f_1(0),
 \eeqq
where $f_1(z)=\int_{\r^{\plus}} e^{\i z x-\alpha x} \pi(x) \d x$.
 \item[(ii)] Assume that $\int_{\r^{\plus}} x e^{-\alpha x} \pi(x) \d x < \infty$. Then for $z\in \r$
 \beqq
  \int\limits_{\r^{\plus}} \left(e^{\i z x}-1-\i z x\right) e^{-\alpha x} \pi'(x) \d x=(\alpha-\i z) f_2(z)-\alpha z f_2'(0),
 \eeqq
where $f_2(z)=\int_{\r^{\plus}} \left(e^{\i z x}-1\right)e^{-\alpha x} \pi(x) \d x$.
 \end{itemize}
\end{proposition}

Asymptotic expansions for the large solutions of the equation $\phi(\zeta)=q$ are obtained using  exactly 
the same technique as in the proof of the Theorem 5 in \cite{Kuznetsov2009}, 
and since the derivation of these expressions is rather lengthy and tedious we have decided to omit it. 
The interested reader who decides to verify these formulas might want to use a symbolic 
computation software to make the algebraic manipulations more enjoyable. 

All the formulas below involve the parameter $\sigma\ge 0$, positive numbers $\alpha_i, \beta_i, c_i$ which define the L\'evy measure $\Pi_{\chi}$ via Definition \ref{def_nu_chi}, and two additional parameters $\gamma$ and $\rho$. The parameter $\gamma$ is uniquely determined by the condition $\Psi(0)=0$ and the parameter $\rho$ is
responsible for the linear drift. In all the cases when $\chi<2$, the jump part of the process has bounded variation, 
thus we can take the cutoff function in the L\'evy-Khintchine formula
(\ref{def_Psi}) as $h(x)\equiv 0$, and then we have $\rho=\mu$. 
When $\chi\ge 2$ we take $h(x)\equiv 1$ and then $\rho$ can be uniquely expressed in terms of the characteristic triplet $(\mu,\sigma,\Pi_{\chi})$ via condition
$\e[X_1]=\i \Psi'(0)=\mu$.

\subsection{The family of processes with $\chi=1/2$}\label{sect12}

\begin{itemize}
 \item[{\bf (i)}] The characteristic exponent can be computed as follows 
\beqq
\Psi(z)&=&\frac12\sigma^2 z^2-\i \rho z-c_1\pi \left((\alpha_1-\i z)\beta_1^{-1}\right)^{-\frac12} \coth\left( \pi \sqrt{(\alpha_1-\i z)\beta_1^{-1}} \right) 
 \\ &-&
c_2\pi \left((\alpha_2+\i z)\beta_2^{-1}\right)^{-\frac12} \coth\left( \pi \sqrt{(\alpha_2+\i z)\beta_2^{-1}} \right) +\gamma.
\eeqq
\item[{\bf (ii)}] For $n\ge 1$ we have  $\rho_{n}=\alpha_1+\beta_1 (n-1)^2$ and $\hat\rho_n=\alpha_2 +\beta_2 (n-1)^2$. 
\item[{\bf (iii)}] If $\sigma\ne 0$ then the large positive solutions to $\phi(\zeta)=q$ satisfy
\beqq
\zeta \sim \beta_2 n^2+\alpha_2+\frac{4}{\sigma^2}\frac{c_2}{\beta_2} n^{-4}+ \frac{8}{\sigma^4}  \frac{c_2}{\beta_2^2}(\mu-\alpha_2\sigma^2)n^{-6}+O\left(n^{-8}\right), \;\;\; n\to +\infty.
\eeqq
\item[{\bf (iv)}] 
If $\sigma=0$ and $\mu \ne 0$ then the large positive solutions to $\phi(\zeta)=q$ satisfy
\beqq
\zeta \sim\beta_2 n^2+\alpha_2-\frac{2c_2}{\mu} n^{-2}+ \frac{2}{\mu^2} \frac{c_2}{\beta_2 }(\mu\alpha_2+\gamma+q)n^{-4}+O\left(n^{-5}\right), \;\;\;
n\to +\infty. 
\eeqq
\end{itemize}

\subsection{The family of processes with $\chi=1$}\label{sect22}

\begin{itemize}
  \item[{\bf (i)}] The characteristic exponent can be computed as follows 
\beqq
\Psi(z)&=&\frac12\sigma^2 z^2-\i \rho z+c_1 \psi\left(\i \sqrt{(\alpha_1-\i z)\beta_1^{-1}} \right)+c_1 \psi\left(-\i \sqrt{(\alpha_1-\i z)\beta_1^{-1}} \right)\\
&+&c_2 \psi\left(\i \sqrt{(\alpha_2+\i z)\beta_2^{-1}} \right)+c_2 \psi\left(-\i \sqrt{(\alpha_2+\i z)\beta_2^{-1}} \right)-\gamma.
\eeqq
\item[{\bf (ii)}] For $n\ge 1$ we have  $\rho_{n}=\alpha_1+\beta_1 n^2$ and $\hat\rho_n=\alpha_2 +\beta_2 n^2$. 
\item[{\bf (iii)}]  If $\sigma\ne 0$ then the large positive solutions to $\phi(\zeta)=q$ satisfy
\beqq
\zeta \sim \beta_2 n^2+\alpha_2+\frac{4}{\sigma^2}\frac{c_2}{\beta_2} n^{-3}+ \frac{8}{\sigma^4}  \frac{c_2}{\beta_2^2}(\mu-\alpha_2\sigma^2)n^{-5}+O\left(n^{-7}\right), \;\;\; n\to +\infty. 
\eeqq
 \item[{\bf (iv)}]  If $\sigma=0$ and $\mu \ne 0$ then the large positive solutions to $\phi(\zeta)=q$ satisfy
\beqq
\zeta \sim\beta_2 n^2+\alpha_2-\frac{2c_2}{\mu} n^{-1}+ \frac{2}{\mu^2} \frac{c_2}{\beta_2}(2(c_1+c_2)\ln(n)+c_0)n^{-3}+O\left(n^{-4}\ln(n)\right), \;\;\;
n\to +\infty, 
\eeqq
where $c_0=\mu\alpha_2-\gamma+q+c_1\ln\left(\frac{\beta_2}{\beta_1}\right)$.
\end{itemize}

\subsection{The family of processes with $\chi=3/2$}\label{sect32}

\begin{itemize}
  \item[{\bf (i)}]  The characteristic exponent can be computed as follows 
\beqq
\Psi(z)&=&\frac12\sigma^2 z^2-\i \rho z+c_1\pi \sqrt{(\alpha_1-\i z)\beta_1^{-1}} \coth\left( \pi \sqrt{(\alpha_1-\i z)\beta_1^{-1}}\right)\\
&+&c_2\pi \sqrt{(\alpha_2+\i z)\beta_2^{-1}} \coth\left( \pi \sqrt{(\alpha_2+\i z)\beta_2^{-1}}\right)-\gamma.
\eeqq
\item[{\bf (ii)}] For $n\ge 1$ we have  $\rho_{n}=\alpha_1+\beta_1 n^2$ and $\hat\rho_n=\alpha_2 +\beta_2 n^2$.
\item[{\bf (iii)}] If $\sigma\ne 0$ then the large positive solutions to $\phi(\zeta)=q$ satisfy
\beqq
\zeta \sim \beta_2 n^2+\alpha_2+\frac{4}{\sigma^2}\frac{c_2}{\beta_2} n^{-2}+ \frac{8}{\sigma^4} \frac{c_2}{\beta_2^2 }(\mu-\alpha_2\sigma^2)n^{-4}+O\left(n^{-6}\right), \;\;\; n\to +\infty.
\eeqq
\item[{\bf (iv)}] 
If $\sigma=0$ and $\mu\ne 0$  then the large positive solutions to $\phi(\zeta)=q$ satisfy
\beqq
\zeta \sim \beta_2 n^2+\alpha_2-\frac{2c_2}{\mu}+\frac{2\pi}{\mu^2} \frac{c_1 c_2}{\sqrt{\beta_1 \beta_2}}n^{-1} +O\left(n^{-2}\right), \;\;\; n\to +\infty. 
\eeqq
\end{itemize}

\subsection{The family of processes with $\chi=2$}\label{sect42}

\begin{itemize}
  \item[{\bf (i)}]  The characteristic exponent can be computed as follows 
\beqq
\Psi(z)&=&\frac12\sigma^2 z^2-\i \rho z-c_1 (\alpha_1-\i z)\beta_1^{-1}  
 \left[ \psi\left(\i \sqrt{(\alpha_1-\i z)\beta_1^{-1}} \right)+\psi\left(-\i \sqrt{(\alpha_1-\i z)\beta_1^{-1}} \right) \right]\\
&-&c_2 (\alpha_2+\i z)\beta_2^{-1}  
 \left[ \psi\left(\i \sqrt{(\alpha_2+\i z)\beta_2^{-1}} \right)+\psi\left(-\i \sqrt{(\alpha_2+\i z)\beta_2^{-1}} \right) \right]+\gamma.
\eeqq
\item[{\bf (ii)}] For $n\ge 1$ we have  $\rho_{n}=\alpha_1+\beta_1 n^2$ and $\hat\rho_n=\alpha_2 +\beta_2 n^2$.
\item[{\bf (iii)}]  If $\sigma\ne 0$ then the large positive solutions to $\phi(\zeta)=q$ satisfy
\beqq
\zeta \sim \beta_2 n^2+\alpha_2+\frac{4}{\sigma^2}\frac{c_2}{\beta_2} n^{-1}+O\left(n^{-3}\ln(n)\right), \;\;\; n\to +\infty. 
\eeqq
 \end{itemize}

\subsection{The family of processes with $\chi=5/2$}\label{sect52}

\begin{itemize}
  \item[{\bf (i)}]  The characteristic exponent can be computed as follows 
\beqq
\Psi(z)&=&\frac12\sigma^2 z^2-\i \rho z-c_1\pi\left((\alpha_1-\i z)\beta_1^{-1}\right)^{\frac32} \coth\left( \pi \sqrt{(\alpha_1-\i z)\beta_1^{-1}}\right)\\
&-&c_2\pi\left((\alpha_2+\i z)\beta_2^{-1}\right)^{\frac32} \coth\left( \pi \sqrt{(\alpha_2+\i z)\beta_2^{-1}}\right)+\gamma.
\eeqq
\item[{\bf (ii)}] For $n\ge 1$ we have  $\rho_{n}=\alpha_1+\beta_1 n^2$ and $\hat\rho_n=\alpha_2 +\beta_2 n^2$.
\item[{\bf (iii)}]  If $\sigma\ne 0$ then the large positive solutions to $\phi(\zeta)=q$ satisfy
\beqq
\zeta \sim \beta_2 n^2+\alpha_2+\frac{4}{\sigma^2}\frac{c_2}{\beta_2 }- \frac{8\pi}{\sigma^4} \frac{c_1c_2}{(\beta_1\beta_2)^{\frac32}}n^{-1}+O\left(n^{-2}\right), \;\;\; n\to +\infty. 
\eeqq
\item[{\bf (iv)}]  If $\sigma=0$ then the large positive solutions to $\phi(\zeta)=q$ satisfy
 \beqq
 \zeta\sim \beta_2 (n+w_0)^2+ \alpha_2+\frac{2\rho}{\pi^2} \frac{c_2\beta_2^2\beta_1^3}{c_1^2\beta_2^3+c_2^2\beta_1^3}+O\left(n^{-1}\right), \;\;\; n\to +\infty, 
 \eeqq
where 
\beqq
w_0=\frac{1}{\pi}\arctan\left(\dfrac{c_2\beta_1^{\frac32}}{c_1\beta_2^{\frac32}} \right).
\eeqq
\end{itemize}

\section{Conclusion}\label{conclusion}

In this paper we have extended results in \cite{Kuznetsov2009} in several directions. First of all, for a very large class of L\'evy processes 
(having infinitely many parameters) we have proved that the Wiener-Hopf factors can be expressed as infinite products of linear factors and that the distribution
of $S_{\ee(q)}$ and $I_{\ee(q)}$ can be identified as an infinite mixture of exponential distributions.
Second, we have introduced
five  eight-parameter families of L\'evy processes, which have a wide range of behavior of small jumps, 
including one family having jumps of finite activitity
(the density of the L\'evy measure has a singularity at zero of order $\chi=\frac12$), 
two families with jumps of infinite activity but finite variation  ($\chi=1$ or $\chi=\frac32$) and two families with jumps of infinite variation 
($\chi=2$ or $\chi=\frac52$). We have also derived precise asymptotic expressions for the large solutions to $\phi(z)=q$, which are very
 useful for numerical computations. 
These five families of processes have rather simple form of the characteristic exponent, especially when $\chi$ is a half-integer, 
in which case $\Psi(z)$ is given in terms of elementary trigonometric functions. This fact and the availability of efficient 
numerical schemes for computing the Wiener-Hopf factors and the distribution of extrema make 
these processes very interesting for Mathematical Modeling, 
in particular in the areas of Mathematical Finance and Insurance Mathematics.

\newpage



\end{document}